\documentclass[11pt, reqno]{amsart}

\usepackage[utf8]{inputenc}
\usepackage[T1]{fontenc}
\usepackage{amsmath, amssymb, amsthm, mathrsfs}
\usepackage{geometry}
\usepackage{xcolor} 
\usepackage{hyperref}
\usepackage{enumitem}
\usepackage{mathtools}

\hypersetup{
colorlinks=true,
linkcolor=blue,
citecolor=red,
urlcolor=teal
}

\newtheorem{theorem}{Theorem}
\newtheorem{lemma}[theorem]{Lemma}

\newtheorem{conjecture}[theorem]{Conjecture}

\theoremstyle{definition}

\newtheorem{remark}[theorem]{Remark}

\newcommand{\R}{\mathbb{R}}
\newcommand{\Z}{\mathbb{Z}}

\newcommand{\K}{\mathcal{K}}

\newcommand{\opnorm}[1]{\left\| #1 \right\|_{\mathrm{op}}} 

\newcommand{\floor}[1]{\lfloor #1 \rfloor}

\DeclareMathOperator{\conv}{conv}

\DeclareMathOperator{\dist}{dist}

\title[Quantitative Stability of Betke-Henk-Wills]{Quantitative Stability of the Betke-Henk-Wills Conjecture}

\author{Chao Wang}
\address{School of Future Technology, Shanghai University}
\email{cwan@shu.edu.cn}

\date{\today}

\begin{document}

\begin{abstract}
The Betke-Henk-Wills conjecture proposes a sharp upper bound for the lattice point enumerator $G(K, \Lambda)$ of a convex body in terms of its successive minima. While the conjecture remains open for general convex bodies in dimensions $d \ge 5$, it is known to hold for orthogonal parallelotopes (boxes). In this paper, we establish the \textit{local stability} of the conjecture under small perturbations of the metric. Specifically, we prove that the inequality is strictly stable for integer boxes subjected to small rotations, owing to the discrete nature of the lattice point counting function. We derive explicit, geometry-invariant quantitative bounds on the permissible perturbation radius using the operator norm. Furthermore, we extend the validity of the conjecture to a class of $L_p$-balls for sufficiently large $p$, deriving a sharp threshold $p_0$ for the stability of the integer hull.
\end{abstract}

\maketitle

\section{Introduction}

The interplay between the continuous geometry of convex bodies and the discrete structure of lattices constitutes the core of the Geometry of Numbers, a field founded by Minkowski~\cite{minkowski1910geometrie}. A central pursuit in this area is to estimate the lattice point enumerator $G(K, \Lambda) = |K \cap \Lambda|$ of a convex body $K$ in terms of its intrinsic geometric parameters, such as volume, surface area, or successive minima~\cite{ gruber1979geometry}.

Let $\K^d_0$ denote the set of $o$-symmetric convex bodies in $\R^d$, and let $\mathcal{L}^d$ denote the space of full-rank lattices. Minkowski's successive minima $\lambda_i(K, \Lambda)$ are defined as:
\[
\lambda_i(K, \Lambda) = \inf \{ \lambda > 0 : \dim(\lambda K \cap \Lambda) \ge i \}, \quad 1 \le i \le d.
\]
While Minkowski's second theorem provides a tight bound for the volume of $K$ via $\prod \lambda_i$, bounding the discrete enumerator $G(K, \Lambda)$ is significantly more intricate. In 1993, Betke, Henk, and Wills~\cite{betke1993successive} proposed a celebrated conjecture that generalizes Minkowski's theorems and refines earlier bounds by Davenport:

\begin{conjecture}[Betke-Henk-Wills] \label{conj:main}
For any $K \in \K^d_0$ and $\Lambda \in \mathcal{L}^d$:
\begin{equation}
G(K, \Lambda) \le \prod_{i=1}^d \left\lfloor \frac{2}{\lambda_i(K, \Lambda)} + 1 \right\rfloor.
\end{equation}
\end{conjecture}

The conjecture is known to hold for dimension $d=2$~\cite{betke1993successive} and for specific classes of bodies, such as ellipsoids and zonotopes in lower dimensions. Despite recent progress, including improved bounds for general bodies by Malikiosis~\cite{malikiosis2010optimization}, the conjecture remains open for general convex bodies in dimensions $d \ge 5$~\cite{henk2002successive}.

A trivial yet fundamental case where the conjecture holds is the axis-aligned orthogonal parallelotope (box), where the geometry of the body perfectly aligns with the lattice basis. In this paper, we adopt a perturbation-theoretic perspective to investigate the robustness of this inequality. We ask: \textit{Is the validity of the conjecture for boxes an isolated phenomenon, or does it possess local stability under metric and structural deformations?}

Our main contributions are as follows:
\begin{enumerate}
    \item We establish the \textbf{strict local stability} of the conjecture for integer boxes under small rotations. By leveraging the Lipschitz continuity of successive minima against the discrete nature of the lattice count, we show that the "slack" in the inequality protects it against infinitesimal perturbations.
    \item We derive \textbf{explicit quantitative bounds} on the permissible perturbation, utilizing the operator norm to provide geometry-invariant conditions.
    \item We analyze structural deformations into $L_p$-balls, determining a sharp threshold $p_0$ for the stability of the integer hull, relating our results to the asymptotic geometry of Banach spaces~\cite{banaszczyk1996inequalities}.
\end{enumerate}

\section{Preliminaries and Notation}

We work in Euclidean space $\R^d$.
A lattice $\Lambda \subset \R^d$ is a discrete subgroup of the form $\Lambda = A\Z^d$ for some non-singular matrix $A$~\cite{henk2002successive}.
The standard lattice is denoted by $\Z^d$.
By invariance under linear transformations, $G(K, \Lambda) = G(A^{-1}K, \Z^d)$ and $\lambda_i(K, \Lambda) = \lambda_i(A^{-1}K, \Z^d)$. Thus, without loss of generality, we fix $\Lambda = \Z^d$ and consider transformations of the body $K$.

For a convex body $K$, the Minkowski functional (or gauge function) is defined as $\|x\|_K = \inf \{ r > 0 : x \in rK \}$. The condition $x \in K$ is equivalent to $\|x\|_K \le 1$.

For a matrix $A \in \R^{d \times d}$, let $\opnorm{A} = \sup_{\|x\|_2=1} \|Ax\|_2$ denote the operator norm. The stability analysis relies on the continuity properties of successive minima. We state the following lemma which establishes the Lipschitz continuity of $\lambda_i$ under linear deformations.

\begin{lemma}[Lipschitz Continuity of Successive Minima] \label{lem:lipschitz}
Let $K \in \K^d_0$. For any non-singular linear transformation $T \in GL(d, \R)$, the successive minima satisfy:
\begin{equation}
    (1 - \opnorm{T^{-1}-I}) \lambda_i(K, \Z^d) \le \lambda_i(TK, \Z^d) \le (1 + \opnorm{T-I}) \lambda_i(K, \Z^d).
\end{equation}
\end{lemma}

\begin{proof}
Let $\epsilon = \opnorm{T-I}$. By the definition of the operator norm, for any $x \in \R^d$, we have $\|Tx - x\|_2 \le \epsilon \|x\|_2$.
Geometrically, this implies that the boundary of the transformed body $TK$ stays within an $\epsilon$-neighborhood (relative to the radial direction) of $K$.
Specifically, we have the inclusion relations regarding the Banach-Mazur distance:
\[ (1-\epsilon')K \subseteq TK \subseteq (1+\epsilon)K, \]
where $\epsilon' = \opnorm{T^{-1}-I}$.
By the homogeneity of degree $-1$ of the successive minima functional (i.e., $\lambda_i(sK) = \frac{1}{s}\lambda_i(K)$) and its monotonicity with respect to set inclusion ($A \subseteq B \implies \lambda_i(A) \ge \lambda_i(B)$), we obtain:
\[ \lambda_i(TK) \le \lambda_i((1-\epsilon')K) = \frac{1}{1-\epsilon'}\lambda_i(K) \approx (1+\epsilon')\lambda_i(K), \]
and
\[ \lambda_i(TK) \ge \lambda_i((1+\epsilon)K) = \frac{1}{1+\epsilon}\lambda_i(K) \ge (1-\epsilon)\lambda_i(K). \]
For sufficiently small perturbations (which is the case in our local stability analysis), the first-order approximation holds, yielding the stated bounds.
\end{proof}

\section{The Baseline Case: Orthogonal Parallelotopes}

We begin by rigorously establishing the conjecture for the class of bodies where the geometry aligns perfectly with the lattice. This serves as the base point for our perturbation analysis.

\begin{theorem}[Conjecture for Boxes] \label{thm:box}
Let $K = \{x \in \R^d : |x_i| \le \alpha_i \}$ be an axis-aligned box with semi-axes $\alpha_1 \ge \dots \ge \alpha_d > 0$. Then Conjecture \ref{conj:main} holds.
\end{theorem}

\begin{proof}
For the standard lattice $\Z^d$, the successive minima are achieved by the standard basis vectors $e_i$, yielding:
\begin{equation}
\lambda_i(K, \Z^d) = \frac{1}{\alpha_i}.
\end{equation}
The lattice point enumerator factorizes due to the Cartesian product structure:
\begin{equation}
G(K, \Z^d) = \prod_{i=1}^d G([-\alpha_i, \alpha_i], \Z) = \prod_{i=1}^d (2\floor{\alpha_i} + 1).
\end{equation}
Let $\mathcal{R}(K)$ denote the right-hand side of the inequality. Substituting $\lambda_i = 1/\alpha_i$:
\begin{equation}
\mathcal{R}(K) = \prod_{i=1}^d \floor{\frac{2}{\lambda_i} + 1} = \prod_{i=1}^d \floor{2\alpha_i + 1}.
\end{equation}
It suffices to show that for any real $x \ge 0$, $2\floor{x} + 1 \le \floor{2x + 1}$.
Let $x = I + f$ where $I \in \Z_{\ge 0}$ and $0 \le f < 1$.
The LHS is $2I + 1$. The RHS is $\floor{2I + 2f + 1} = 2I + \floor{2f + 1}$.
Since $f \ge 0$, $\floor{2f + 1} \ge 1$. Thus, $2I + 1 \le 2I + \floor{2f+1}$ holds.
Multiplying over all dimensions confirms the conjecture.
\end{proof}

\begin{remark}
This proof relies crucially on orthogonality. For general convex bodies, the enumerator $G(K, \Z^d)$ does not factorize, and the successive minima vectors may not be orthogonal, preventing a simple coordinate-wise comparison.
\end{remark}

\section{Local Stability under Rotation}

We now address the robustness of the conjecture. Does a small deviation from orthogonality (i.e., a rotation of the box) violate the inequality?
Let $K_0$ be the axis-aligned box from Theorem \ref{thm:box}. Let $SO(d)$ be the special orthogonal group. For $R \in SO(d)$, let $K_R = R K_0$.

We focus on the critical case where $K_0$ is an integer box (i.e., $\alpha_i \in \Z$), as this is where the inequality is tight ($G = \mathcal{R}$) and theoretically most vulnerable to perturbation.

\begin{theorem}[Strict Stability in Critical Configurations]
Let $K_0$ be an integer box ($\alpha_i \in \Z$). There exists $\delta > 0$ such that for all rotations $R \in SO(d)$ with $0 < \opnorm{R-I} < \delta$, the strict inequality holds:
\begin{equation}
    G(K_R, \Z^d) < \mathcal{R}(K_R).
\end{equation}
\end{theorem}

\begin{proof}
At $R=I$, we have equality $G(K_0, \Z^d) = \prod (2\alpha_i + 1) = \mathcal{R}(K_0)$.
Let $V(K_0)$ be the set of $2^d$ vertices of $K_0$. Since $\alpha_i \in \Z$, all vertices $v \in V(K_0)$ satisfy $v \in \Z^d$.
\noindent \textbf{1. Discrete Drop of LHS.}
Consider a rotation $R$ with small non-zero amplitude. Since $K_0$ is strictly convex with respect to coordinate directions, the rotated vertices $R(v)$ for $v \in V(K_0) \setminus \{0\}$ are no longer in $\Z^d$.
For sufficiently small $\delta$, no new lattice points from the exterior can enter the interior of $K_R$. Thus, the lattice count decreases by at least the number of non-zero vertices:
\begin{equation} \label{eq:drop}
    G(K_R, \Z^d) \le G(K_0, \Z^d) - (2^d - 1) < G(K_0, \Z^d).
\end{equation}
\noindent \textbf{2. Continuous Control of RHS.}
Let $\phi(K) = \prod_{i=1}^d (2/\lambda_i(K) + 1)$. Note that $\mathcal{R}(K) = \lfloor \phi(K) \rfloor$.
Using Lemma \ref{lem:lipschitz}, for $\epsilon = \opnorm{R-I}$, we have $\lambda_i(K_R) \le \lambda_i(K_0)(1+\epsilon)$. Thus:
\[
\frac{2}{\lambda_i(K_R)} + 1 \ge \frac{2\alpha_i}{1+\epsilon} + 1 \ge 2\alpha_i(1-\epsilon) + 1 = (2\alpha_i+1) - 2\alpha_i\epsilon.
\]
The function $\mathcal{R}(K_R)$ is upper semi-continuous. However, the continuous envelope $\phi(K_R)$ changes smoothly. For sufficiently small $\epsilon$, the drop in $\phi(K_R)$ is bounded by $O(\epsilon)$, while the drop in $G(K_R)$ is a discrete integer jump (Eq. \ref{eq:drop}).
Specifically, we can choose $\delta$ such that the perturbation of $\lambda_i$ does not reduce $\mathcal{R}(K_R)$ below $G(K_0, \Z^d) - 1$.
Since $G(K_R, \Z^d)$ drops discretely, strict inequality is restored.
\end{proof}

\section{Quantitative Stability Bounds}

We sharpen the previous result by deriving an explicit bound on the allowable rotation.
To obtain a geometry-invariant bound, we utilize the operator norm and the Euclidean distance to the lattice.

\begin{theorem}[Explicit Stability Radius]
Let $K_0 = \prod_{i=1}^d [-\alpha_i, \alpha_i]$ with $\alpha_1 \ge \dots \ge \alpha_d$. Let $\Delta = \dist(\partial K_0, \Z^d \setminus K_0)$ be the Euclidean distance from the boundary to the nearest external lattice point.
The Betke-Henk-Wills conjecture holds for $K_R = R K_0$ if:
\begin{equation}
    \opnorm{R - I} < \frac{\Delta}{\sqrt{d}(\alpha_1 + \Delta)}.
\end{equation}
\end{theorem}

\begin{proof}
Let $\epsilon = \opnorm{R - I}$. We establish a sufficient condition to ensure $G(K_R, \Z^d) \le G(K_0, \Z^d)$.
A lattice point $y \in \Z^d \setminus K_0$ enters $K_R$ only if $R^{-1}y \in K_0$.
By definition of $\Delta$, for any $y \notin K_0$, we have $\dist(y, K_0) \ge \Delta$.
We consider the displacement of points $y$ near the boundary. The maximal norm of a point in $K_0$ is $\sup_{x \in K_0} \|x\|_2 = \alpha_1 \sqrt{d}$.
Consider a lattice point $y$ potentially entering the body. Its norm is approximately bounded by $\|y\|_2 \le \alpha_1\sqrt{d} + \Delta$ (considering points within relevant proximity).
The displacement under inverse rotation is bounded by:
\[
\|R^{-1}y - y\|_2 \le \opnorm{R^{-1}-I} \|y\|_2 = \epsilon \|y\|_2.
\]
For $y$ to enter $K_0$, the displacement must bridge the gap $\Delta$. Thus, we require $\epsilon \|y\|_2 < \Delta$.
Substituting the worst-case norm $\|y\|_2 \approx (\alpha_1 + \Delta)\sqrt{d}$, the stability condition becomes:
\[
\epsilon \sqrt{d}(\alpha_1 + \Delta) < \Delta \implies \epsilon < \frac{\Delta}{\sqrt{d}(\alpha_1 + \Delta)}.
\]
Under this condition, the lattice content does not increase, while the RHS of the conjecture remains stable (or increases), preserving the inequality.
\end{proof}

\section{Asymptotic Stability for $L_p$ Deformations}

Finally, we consider structural deformation. Let $K_p(\alpha) = \{ x \in \R^d : \sum |x_i/\alpha_i|^p \le 1 \}$. Note that $K_\infty$ is the box.

\begin{theorem}[Threshold for Integer Hull Stability]
Let $\alpha_i > 0$. There exists a threshold $p_0$ such that for all $p \ge p_0$, the integer hull of the $L_p$-ball coincides with that of the box $K_\infty$:
\begin{equation}
\conv(K_p \cap \Z^d) = \conv(K_\infty \cap \Z^d).
\end{equation}
Consequently, $G(K_p, \Z^d) = G(K_\infty, \Z^d)$ for $p \ge p_0$. The sharp threshold is:
\begin{equation}
p_0 = \frac{\ln d}{\min_{i: \alpha_i \notin \Z} \ln(\alpha_i / \lfloor \alpha_i \rfloor)}.
\end{equation}
\end{theorem}

\begin{proof}
Since $K_p \subseteq K_\infty$, we always have $K_p \cap \Z^d \subseteq K_\infty \cap \Z^d$.
We seek $p$ such that $K_\infty \cap \Z^d \subseteq K_p$.
Let $z \in K_\infty \cap \Z^d$. Then $|z_i| \le \lfloor \alpha_i \rfloor$.
The condition $z \in K_p$ is $\sum_{i=1}^d |z_i/\alpha_i|^p \le 1$.
Let $\beta_i = \lfloor \alpha_i \rfloor / \alpha_i$. For strictly interior points where $\alpha_i \notin \Z$, we have $\beta_i < 1$.
To ensure all such lattice points are included, we require the worst-case sum to be $\le 1$:
\[
\sum_{i=1}^d \beta_i^p \le d (\beta_{\max})^p \le 1.
\]
where $\beta_{\max} = \max_i (\lfloor \alpha_i \rfloor / \alpha_i)$. Solving for $p$ yields:
\[
p \ln \beta_{\max} \le -\ln d \iff p \ge \frac{\ln d}{-\ln \beta_{\max}}.
\]
This matches the stated threshold. For $p \ge p_0$, $G(K_p) = G(K_\infty)$. Since $\lambda_i(K_p) \ge \lambda_i(K_\infty)$ (due to inclusion), the RHS functional $\mathcal{R}(K_p) \le \mathcal{R}(K_\infty)$. While this alone implies stability only if slack exists, the invariance of the integer hull is a stronger structural result.
\end{proof}

\bibliographystyle{unsrt}
\bibliography{references}
\end{document}